\documentstyle[amsfonts, txfonts, amssymb, epsfig, latexsym, 12pt]{amsart}

\newtheorem{theorem}{Theorem}[section]
\newtheorem{lemma}[theorem]{Lemma}

\newtheorem{definition}[theorem]{Definition}

\def\Z{{\mbox{\rm\kern.25em
\vrule width.03em height0.57ex depth0ex
\kern.033em
\vrule width.03em height1.52ex depth-0.96ex \kern-.338em Z}}}
\def\R{{\mbox{\rm I\kern-.22em R}}}
\def\N{{\mbox{\rm I\kern-.22em N}}}

\def\D{{\bf D}}
\def\T{{\bf T}}
\def\supp{{\rm supp}}

\def\sgn{{\rm sgn}}

\def\n{{\bf n}}

\def\D{{\cal{D}}}

\def\I{{\cal{I}}}

\def\dist{{\rm dist}}

\def\111{\gamma}

\def\be#1{\begin{equation}\label{#1}}
\def\bas{\begin{align*}}
\def\eas{\end{align*}}
\def\bi{\begin{itemize}}
\def\ei{\end{itemize}}
\newenvironment{proof}{\noindent {\bf Proof} }{\endprf\par}
\def \endprf{\hfill  {\vrule height6pt width6pt depth0pt}\medskip}
\def\emph#1{{\it #1}}

\title[Calder\'{o}n commutators and Cauchy integral I]{Calder\'{o}n commutators and the Cauchy integral on Lipschitz curves revisited I. First commutator and generalizations}

\author{Camil Muscalu}
\address{Department of Mathematics, Cornell University, Ithaca, NY 14853}
\email{camil@@math.cornell.edu}

\begin{document}

\begin{abstract}
This article is the first in a series of three papers, whose scope is to give new proofs to the well known theorems of Calder\'{o}n, Coifman, McIntosh and Meyer \cite{calderon}, \cite{meyerc}, \cite{cmm}. Here we treat the case of
the first commutator of Calder\'{o}n and some of its generalizations.
\end{abstract}

\maketitle

\section{Introduction}

This is the first paper in a sequel of three, whose aim is to give new proofs to the well known theorems of Calder\'{o}n, Coifman, McIntosh and Meyer \cite{calderon}, \cite{meyerc}, \cite{cmm}, which established $L^p$ estimates for the so called Calder\'{o}n
commutators and the Cauchy integral on Lipschitz curves. 

We refer the reader to the book of Coifman and Meyer \cite{meyerc} for a description of the history of these fundamental analytical objects, the role they play in analysis and the various
methods that have been further developed to understand these operators, since the appearance of the original articles. 

Other expository papers where some of these results are described and connected with other parts of mathematics, 
are the proceedings of the plenary talks at the ICM 1974 in Vancouver and the
ICM 1978 in Helsinki, given by Fefferman \cite{fefferman} and Calder\'{o}n \cite{calderon1}.

Our approach will also turn out to be flexible and generic enough, to allow us to generalize these classical results in various new ways.

The first paper describes the case of the first commutator and its generalizations, the second one the case of the Cauchy integral on Lipschitz curves and its generalizations
and finally, the third will be devoted to the extension of all these results to the multi-parameter setting of polydiscs of arbirary dimension, solving completely along the way an open question of Coifman from the early eighties.

We naturally start with the first commutator.

Given a Lipschitz function $A$ on the real line (so $A':= a\in L^{\infty}(\R)$) one formally defines the linear operator $C_1(f)$ by the formula

\begin{equation}\label{1}
C_1(f)(x) = p.v. \int_{\R}\frac{A(x)-A(y)}{(x-y)^2} f(y) dy
\end{equation}
where the meaning of the principal value integral is

\begin{equation}\label{2}
\lim_{\epsilon \rightarrow 0} \int_{\epsilon < |x-y| < 1/\epsilon }\frac{A(x)-A(y)}{(x-y)^2} f(y) dy
\end{equation}
whenever the limit exists. This is the so called first commutator of Calder\'{o}n. Note that the simplest particular case is obtained when $A(x)=x$ and $C_1(f)$ becomes the classical Hilbert transform.

Let us observe that when $a$ and $f$ are Schwartz functions, then (\ref{2}) makes perfect sense.

Indeed, for a fixed $\epsilon > 0$, one can rewrite the corresponding expression in (\ref{2}) as 

$$-\int_{\epsilon < |t| < 1/\epsilon }\frac{A(x+t)-A(x)}{t^2}f(x+t) dt = 
-\int_{\epsilon < |t| < 1/\epsilon }\left[\frac{A(x+t)-A(x)}{t}\right]f(x+t) \frac{dt}{t} = $$

\begin{equation}\label{3}
-\int_{\epsilon < |t| < 1/\epsilon }\left[\int_0^1 a(x+\alpha t) d \alpha \right] f(x+t) \frac{dt}{t}.
\end{equation}
Then, write $a$ and $f$ as 

$$a(x+\alpha t) = \int_{\R}\widehat{a}(\xi_1) e^{2\pi i (x+\alpha t)\xi_1} d\xi_1$$
and

$$f(x+t) = \int_{\R}\widehat{f}(\xi) e^{2\pi i (x+t) \xi} d\xi.$$
Using these formulas in (\ref{3}), the expression becomes

\begin{equation}\label{4}
-\int_{\R^2} m_{\epsilon}(\xi, \xi_1) \widehat{f}(\xi) \widehat{a}(\xi_1) e^{2\pi i x (\xi +\xi_1)} d\xi d \xi_1
\end{equation}
where

$$m_{\epsilon}(\xi, \xi_1) = \int_0^1 \int_{\epsilon < |t| < 1/\epsilon }\frac{1}{t} e^{2\pi i t (\xi +\alpha\xi_1)} d t d \alpha$$
which is known to converge uniformly to

$$-\int_0^1 \sgn(\xi +\alpha\xi_1) d\alpha.$$
In particular, the dominated convergence theorem implies that the limit as $\epsilon \rightarrow 0$ exists in (\ref{4}) and it is equal to

\begin{equation}\label{5}
\int_{\R^2}
\left[\int_0^1 \sgn(\xi +\alpha\xi_1) d\alpha\right]
\widehat{f}(\xi) \widehat{a}(\xi_1) e^{2\pi i x (\xi +\xi_1)} d\xi d \xi_1.
\end{equation}
Because of this formula, one can think of $C_1$ as being a bilinear operator in $f$ and $a$ and we will denote it from now on with $C_1(f, a)$. The following theorem of Calder\'{o}n is classical \cite{calderon}.

\begin{theorem}\label{unu}
For every $A'=a\in L^{\infty}$, the operator $C_1$ extends naturally as a bounded linear operator from $L^p(\R)$ into $L^p(\R)$ for every $1<p<\infty$, satisfying

\begin{equation}\label{6}
\|C_1(f, a)\|_p \lesssim \|a\|_{\infty}\cdot\|f\|_p.
\end{equation} 
\end{theorem}
At this point, one should observe that the symbol of (\ref{5}) given by

\begin{equation}\label{symbol}
(\xi, \xi_1) \rightarrow \int_0^1 \sgn(\xi +\alpha\xi_1) d\alpha
\end{equation}
is not a Marcinkiewicz-H\"{o}rmander-Mihlin symbol \cite{stein} and as a consequence, the Coifman-Meyer theorem on paraproducts \cite{meyerc} cannot be applied. More precisely, one can see that away from the lines
$\xi=0$ and $\xi+\xi_1=0$, the symbol (\ref{symbol}) is many times differentiable and behaves like a classical symbol, but along them it is only continuous. The observation on which our approach is based, is that in spite of
this lack of differentiability, when one {\it smoothly} restricts (\ref{symbol}) to an arbitrary Whitney square with respect to the origin
\footnote{These are squares whose sides are parallel to the coordinate axes and whose distance to the origin is comparable to their sidelengths. }, the Fourier coefficients of the corresponding function decay at least quadratically. This fact (which will be proved
carefully later on) will reduce the problem to the one of proving estimates for the associated bilinear operators, which do not grow too fast with respect to the indices of the Fourier coefficients. We will see that this upper bounds can grow at most 
logarithmically, which will be more than enough to make the final power series convergent. This is, in just a few words, our strategy of the proof.

\begin{figure}[htbp]\centering
\psfig{figure=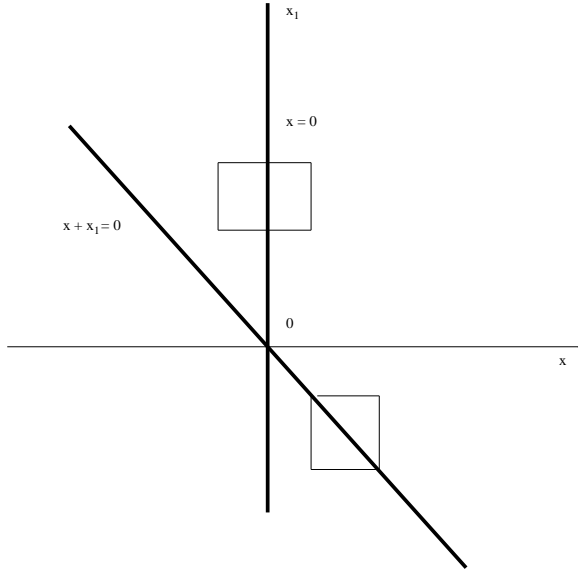, height=3in, width=3in}
\caption{The singularities of the symbol of the first commutator}
\label{fig1}
\end{figure}

Before going any further, let us also remark that if one permutes the two integrations in (\ref{5}), one can rewrite that expression as

$$\int_0^1 BHT_{\alpha}(f, a) (x) d \alpha$$
where $BHT_{\alpha}$ is the so called bilinear Hilbert transform of parameter $\alpha$. An alternative approach to the first commutator (suggested by Calder\'{o}n), was to prove $L^p\times L^{\infty} \rightarrow L^p$ estimates for these
operators, with implicit constants that are integrable or even uniform in $\alpha$. Estimates for the bilinear Hilbert transform have been first proved by Lacey and Thiele in \cite{lt1} and \cite{lt2},  and uniform estimates have been later on 
obtained by Thiele \cite{thiele}, Grafakos and Li \cite{gli} and Li \cite{li}. It is also interesting to remark that it is not yet known whether such an approach would work for 
the second Calder\'{o}n commutator \footnote{The second commutator can similarly be seen as a trilinear operator with symbol $\int_{[0,1]^2}\sgn(\xi + \alpha \xi_1 + \beta \xi_2) d \alpha d \beta$.}
which this time can be written as

$$\int_{[0,1]^2} THT_{\alpha, \beta} (f, a, a)(x) d \alpha d \beta.$$
A recent result of Palsson \cite{eyvi} proved many estimates for the operator $\int_0^1 THT_{\alpha, \beta} d \alpha$ ($\beta$ is fixed now), but so far there are no $L^p$ estimates available (uniform or not) for the corresponding trilinear 
operator $THT_{\alpha, \beta}$ which has been named by several authors, the trilinear Hilbert transform.

Now coming back to (\ref{6}), in order to describe the way in which $C_1(f,a)$ can be defined for any $a\in L^{\infty}$ and any $f\in L^p(\R)$, we need to say a few words about adjoints of bilinear operators.

If $m(\xi_1,\xi_2)$ is a bounded symbol, denote by $T_m(f_1, f_2)$ the bilinear operator given by

\begin{equation}\label{7}
T_m(f_1, f_2)(x) = \int_{\R^2} m(\xi_1, \xi_2)
\widehat{f_1}(\xi_1)
\widehat{f_2}(\xi_2)
e^{2\pi i x(\xi_1 + \xi_2)}
d\xi_1
d\xi_2,
\end{equation}
for $f_1, f_2$ Schwartz functions. Associated with it is the trilinear form $\Lambda(f_1, f_2, f_3)$ defined by

$$\Lambda(f_1, f_2, f_3) = \int_{\R} T_m(f_1, f_2)(x) f_3(x) d x,$$
again for $f_1, f_2, f_3$ Schwartz functions.

There are two adjoint operators $T_m^{\ast 1}$ and $T_m^{\ast 2}$ naturally defined by the equalities

$$\int_{\R}
T_m^{\ast 1}(f_2, f_3)(x) f_1(x) d x = \Lambda (f_1, f_2, f_3)
$$
and 

$$\int_{\R}
T_m^{\ast 2}(f_1, f_3)(x) f_2(x) d x = \Lambda (f_1, f_2, f_3)
$$
respectively. It is very easy to observe that both of them are also bilinear multipliers whose symbols are $m(-\xi_1-\xi_2, \xi_2)$ and $m(\xi_1,-\xi_1-\xi_2)$ respectively.

Now, if $a$  and $f$ are Schwartz functions, the inequality (\ref{6}) is equivalent to

\begin{equation}\label{8}
\left|\int_{\R}C_1(f, a)(x) g(x) d x \right| \lesssim \|a\|_{\infty}\cdot \|f\|_p\cdot \|g\|_{p'}
\end{equation}
for any Schwartz function $g$, where $p'$ is the dual index of $p$ (so $1/p + 1/p' = 1$). We also know from the above that

\begin{equation}\label{9}
\int_{\R}C_1(f, a )(x) g(x) d x = \int_{\R}C_1^{\ast 2}(f , g )(x) a(x) d x.
\end{equation}
We are going to prove in the rest of the paper that

\begin{equation}\label{10}
\|C_1^{\ast 2}(f , g )\|_1 \lesssim \|f\|_p\cdot \|g\|_{p'} 
\end{equation}
for any Schwartz functions $f, g$ and this shows that $C_1^{\ast 2}$ can be extended by density to the whole $L^p\times L^{p'}$. But then, this means that the right hand side of (\ref{9}) makes sense for any
$a\in L^{\infty}$, not only for bounded Schwartz functions, and this suggests to extend $C_1(f, a)$ by duality. More specifically, for $f\in L^p$ and $a\in L^{\infty}$, one can define $C_1(f, a)$ to be the unique
$L^p$ function satisfying (\ref{9}) for any $g\in L^{p'}$.

This discussion also proves that to demonstrate Theorem \ref{unu}, we only need to prove (\ref{10}). The idea now is to discretize $C_1^{\ast 2}$ and reduce (\ref{10}) to a discrete finite model.

{\bf Acknowledgement:} We wish to thank Eli Stein who after a talk we gave in Pisa, kindly pointed out to us that the maximal Theorem \ref{maximal} that will enter the picture later on, was actually known
and can be found in \cite{stein} Chapter II.

The present work has been partially supported by the NSF. 
\section{Reduction to a finite localized model}

We start with some standard notations and definitions. An interval $I$ on the real line $\R$ is called dyadic if it is of the form $I=[2^k n, 2^k (n+1)]$ for some $k, n\in\Z$. We will denote by
$\D$ the set of all such dyadic intervals.

If $I\in\D$, we say that a smooth function $\Phi_I$ is a bump adapted to $I$ if and only if one has

$$|\partial^{\alpha}(\Phi_I)(x)|\leq C_{\alpha, N}\cdot \frac{1}{|I|^{\alpha}}\cdot\frac{1}{\left(1+\frac{\dist(x, I)}{|I|}\right)^N}$$
for every integer $N$ and sufficiently many derivatives $\alpha$, where $|I|$ is the length of $I$. If $\Phi_I$ is a bump adapted to $I$, we say that $|I|^{-1/p}\Phi_I$ is an $L^p$ - normalized bump
adapted to $I$, for $1\leq p\leq\infty$. Also, if $I\in\D$ and $\n\in\Z$ we denote by $I_{\n}$ the new dyadic interval $[2^k (n-\n), 2^k (n+1-\n)]$ sitting $\n$ units of length $|I|$ away from $I$.

\begin{definition}\label{fi}
A sequence of $L^2$ - normalized bumps $(\Phi_I)_I$ adapted to dyadic intervals $I$ is said to be of  $\phi$ type, if and only if for each $I$ there exists an interval $\omega_I (=\omega_{|I|})$ symmetric
with respect to the origin, so that $\supp \widehat{\Phi_I}\subseteq \omega_I$ and with $|\omega_I|\sim |I|^{-1}$.
\end{definition}

\begin{definition}\label{psi}
A sequence of $L^2$ - normalized bumps $(\Phi_I)_I$ adapted to dyadic intervals $I$ is said to be of $\psi$ type, if and only if for each $I$ there exists an interval $\omega_I (=\omega_{|I|})$
 so that $\supp \widehat{\Phi_I}\subseteq \omega_I$ and with $|\omega_I|\sim |I|^{-1} \sim \dist (0, \omega_I)$.
\end{definition}

Fix now $\n_1, \n_2$ two integers and $\I\subseteq\D$ a finite arbitrary collection of dyadic intervals. Consider also three sequences of $L^2$ - normalized bumps
$(\Phi^1_{I_{\n_1}})_{I\in\I}$, $(\Phi^2_{I_{\n_2}})_{I\in\I}$, $(\Phi^3_I)_{I\in\I}$ adapted to $I_{\n_1}$, $I_{\n_2}$ and I respectively, such that at least two of them are of $\psi$ type.

The following theorem holds

\begin{theorem}\label{discrete}
The bilinear operator defined by

\begin{equation}\label{11}
T_{\I}(f, g):=
\sum_{I\in\I}
\frac{1}{|I|^{1/2}}
\langle f, \Phi^1_{I_{\n_1}} \rangle
\langle g, \Phi^2_{I_{\n_2}} \rangle
\Phi^3_I
\end{equation}
is bounded from $L^p\times L^q\rightarrow L^r$ for any $1 < p, q < \infty$ and $0<r<\infty$ so that $1/p+1/q=1/r$, with a bound  of type

$$O(\log <\n_1>  \log <\n_2>)$$
depending also implicitly on $p, q$ but independent on the cardinality of $\I$ and on the families of bumps considered (here $<\n>$ simply means $2+|\n|$). 

\end{theorem}

As we will see, this Theorem \ref{discrete} lies at the heart of our estimates and in the rest of the section we will show how it implies the desired inequality (\ref{10}). The idea will be to discretize $C_1^{\ast 2}$ and show that it can be 
reduced to operators of type (\ref{11}). Equivalently, since they both have the same trilinear form, we will discretize $C_1$ instead. We start with two Littlewood-Paley decompositions and write

$$1(\xi) = \sum_{k_1} \widehat{\Psi_{k_1}}(\xi)$$
and

$$1(\xi_1) = \sum_{k_2} \widehat{\Psi_{k_2}}(\xi_1)$$
where as usual, $\widehat{\Psi_{k_1}}(\xi)$ and $\widehat{\Psi_{k_2}}(\xi_1)$ are supported in the regions $|\xi|\sim 2^{k_1}$ and $|\xi_1|\sim 2^{k_2}$ respectively. In particular, we get

\begin{equation}\label{12}
1(\xi, \xi_1) = \sum_{k_1, k_2} \widehat{\Psi_{k_1}}(\xi) \widehat{\Psi_{k_2}}(\xi_1).
\end{equation}
By splitting (\ref{12}) over the regions where $k_1<<k_2$, $k_2<<k_1$ and $k_1\sim k_2$ we obtain the final decomposition

\begin{equation}\label{13}
1(\xi,\xi_1) = \sum_k \widehat{\Phi_k}(\xi) \widehat{\Psi_k}(\xi_1) +
\end{equation}

\begin{equation}\label{14}
\sum_k \widehat{\Psi_k}(\xi) \widehat{\Phi_k}(\xi_1) +
\end{equation}

\begin{equation}\label{15}
\sum_{k_1\sim k_2} \widehat{\Psi_{k_1}}(\xi) \widehat{\Psi_{k_2}}(\xi_1).
\end{equation}

By inserting this into (\ref{5}), $C_1(f, a)$ splits as a sum of three different expressions. It is easy to see that the symbol of the one corresponding to (\ref{14}) is a classical symbol and for this part
the inequality (\ref{10}) follows from the Coifman-Meyer theorem on paraproducts \cite{meyerc}. We are thus left with understanding the other two terms. Notice that the first one (corresponding to (\ref{13})) interacts with the line
$\xi=0$, while the third one (corresponding to (\ref{15})) interacts with the line $\xi+\xi_1=0$ along which the original symbol

$$\int_0^1 \sgn(\xi+\alpha\xi_1) d \alpha$$
is only continuous. Also, for simplicity, from now on we will replace $\int_0^1 \sgn(\xi+\alpha\xi_1) d \alpha$ with $\int_0^1 1_{\R_+}(\xi+\alpha\xi_1) d \alpha$ since the difference of the corresponding operators
is just the product of $a$ and $f$ which clearly satisfies the original H\"{o}lder type inequalities.

Let us now fix a parameter $k\in\Z$ and consider the corresponding expressions (also, since $k_1\sim k_2$ we assume that they are equal, for simplicity). Their trilinear forms are given by

$$\int_{\xi+\xi_1+\xi_2=0}
\left[\int_0^1 1_{\R_+}(\xi+\alpha\xi_1) d \alpha\right]
\widehat{\Phi_k}(\xi) \widehat{\Psi_k}(\xi_1) \widehat{\Psi_k}(\xi_2)
\widehat{f}(\xi)
\widehat{g}(\xi_1)
\widehat{h}(\xi_2)
d\xi
d\xi_1
d\xi_2
$$
and

$$\int_{\xi+\xi_1+\xi_2=0}
\left[\int_0^1 1_{\R_+}(\xi+\alpha\xi_1) d \alpha\right]
\widehat{\Psi_k}(\xi) \widehat{\Psi_k}(\xi_1) \widehat{\Phi_k}(\xi_2)
\widehat{f}(\xi)
\widehat{g}(\xi_1)
\widehat{h}(\xi_2)
d\xi
d\xi_1
d\xi_2
$$
respectively. Clearly, the functions $\widehat{\Psi_k}(\xi_2)$ and $\widehat{\Phi_k}(\xi_2)$ have been inserted naturally into the above expressions (the first are supported away from zero while the support of the second contains the origin).

Now, on the support of $\widehat{\Phi_k}(\xi) \widehat{\Psi_k}(\xi_1)$, the function $\int_0^1 1_{\R_+}(\xi+\alpha\xi_1) d \alpha$ can be written as a double Fourier series of type

\begin{equation}\label{16}
\sum_{n,n_1} C^k_{n,n_1} e^{2\pi i \frac{n}{2^k}\xi} e^{2\pi i \frac{n_1}{2^k}\xi_1}.
\end{equation}
Similarly, on the support of $\widehat{\Psi_k}(\xi) \widehat{\Psi_k}(\xi_1)$ the same function can also be written as 

\begin{equation}\label{17}
\sum_{n,n_1} \widetilde{C^k_{n,n_1}} e^{2\pi i \frac{n}{2^k}\xi} e^{2\pi i \frac{n_1}{2^k}\xi_1}.
\end{equation}

The following Lemma will be crucial and gives upper bounds for these Fourier coefficients.

\begin{lemma}\label{coef}
One has

$$|C^k_{n,n_1}| \lesssim \frac{1}{<n>^2}\cdot\frac{1}{<n_1>^{\#}}$$
and also

$$|\widetilde{C^k_{n,n_1}}| \lesssim \frac{1}{<n>^2}\cdot\frac{1}{<n- n_1>^{\#}} + \frac{1}{<n>^{\#}}\cdot\frac{1}{<n_1>^{\#}}$$
for a fixed large integer $\#$, uniformly in $k$.
\end{lemma}

We will prove this Lemma \ref{coef} at the end of this section. Roughly speaking, it shows that one has at least quadratic decay for all the Fourier coefficients. 

Now, (\ref{16}) produces expressions of type

$$\int_{\xi+\xi_1+\xi_2=0}
\left[\widehat{\Phi_k}(\xi) e^{2\pi i \frac{n}{2^k}\xi}\right] 
\left[\widehat{\Psi_k}(\xi_1) e^{2\pi i \frac{n_1}{2^k}\xi_1}\right]
\widehat{\Psi_k}(\xi_2)
\widehat{f}(\xi)
\widehat{g}(\xi_1)
\widehat{h}(\xi_2)
d\xi
d\xi_1
d\xi_2
$$
which can be rewritten as

$$\int_{\R}(f\ast \Phi^{1,n}_k)(x) (g\ast \Psi^{2, n_1}_k)(x) (h\ast \Psi^3_k)(x) d x$$
and this can be further discretized by standard arguments as an average of expressions of type

\begin{equation}\label{18}
\sum_{|I|= 2^{-k}}
\frac{1}{|I|^{1/2}}
\langle f, \Phi^1_{I_n} \rangle
\langle g, \Phi^2_{I_{n_1}} \rangle
\Phi^3_I,
\end{equation}
where the functions $\Phi^2_{I_{n_1}}$ and $\Phi^3_I$ are of $\psi$ type.

Similarly, (\ref{17}) produces expressions of type

$$\int_{\xi+\xi_1+\xi_2=0}
\left[\widehat{\Psi_k}(\xi) e^{2\pi i \frac{n}{2^k}\xi}\right] 
\left[\widehat{\Psi_k}(\xi_1) e^{2\pi i \frac{n_1}{2^k}\xi_1}\right]
\widehat{\Phi_k}(\xi_2)
\widehat{f}(\xi)
\widehat{g}(\xi_1)
\widehat{h}(\xi_2)
d\xi
d\xi_1
d\xi_2
$$
and as we have seen, these can be further rewritten and discretized again in the form (\ref{18}), where this time  $\Phi^1_{I_{n}}$ and $\Phi^2_{I_{n_1}}$ are of $\psi$ type.
The connection with (\ref{11}) should be clear by now. If we add all the expressions of type (\ref{18}) for all the scales $k\in\Z$, one obtains a discrete trilinear form corresponding to the part
of $C_1$ related to (\ref{13}) (and of course, as we discussed, there is a similar one related to (\ref{15})). In particular, since we are interested in estimating $C_1^{\ast 2}$, its bilinear model is of the form

$$
\sum_{I\in\I}
\frac{1}{|I|^{1/2}}
\langle f, \Phi^1_{I_{n}} \rangle
\langle h, \Phi^2_{I} \rangle
\Phi^3_{I_{n_1}}
$$
which should be rewritten as

$$
\sum_{I\in\I}
\frac{1}{|I|^{1/2}}
\langle f, \Phi^1_{I_{n-n_1}} \rangle
\langle h, \Phi^2_{I_{-n_1}} \rangle
\Phi^3_{I}
$$
to be able to compare it better with (\ref{11}).

Now, using the fact that $C_1^{\ast 2}(f,g)$ makes perfect sense for Schwartz functions (in fact, it can be written as an expression similar to (\ref{5})) and by triangle inequality, Fatou's lemma and Theorem \ref{discrete}, it follows
that

$$\|C_1^{\ast 2}(f,g)\|_1 \lesssim
\sum_{n,n_1}
\sup_k (|C^k_{n,n_1}|, |\widetilde{C^k_{n,n_1}}|)\cdot
\log <n-n_1>\cdot \log <n_1> \cdot \|f\|_p\cdot \|g\|_{p'}$$
which is clearly bounded by $\|f\|_p\cdot \|g\|_{p'}$ as a consequence of the quadratic decay in Lemma \ref{coef}.
This completes the proof of (\ref{10}).

We now describe the proof of Lemma \ref{coef}. We first record the following

\begin{lemma}\label{derivative}
One has the following identities

(a) $\partial_{\xi}^2\left(\int_0^{\xi_1} 1_{\R_+}(\xi + \alpha) d\alpha \right) = \delta_0(\xi +\xi_1) - \delta_0(\xi)$.

(b) $\partial_{\xi_1}^2\left(\int_0^{\xi_1} 1_{\R_+}(\xi + \alpha) d\alpha \right) = \delta_0(\xi +\xi_1)$.
 
(c) $\partial_{\xi}\partial_{\xi_1}\left(\int_0^{\xi_1} 1_{\R_+}(\xi + \alpha) d\alpha \right) = \partial_{\xi_1}\partial_{\xi}\left(\int_0^{\xi_1} 1_{\R_+}(\xi + \alpha) d\alpha \right) = \delta_0(\xi +\xi_1)$.
 
where $\delta_0$ is the Dirac distribution with respect to the origin.
\end{lemma}

\begin{proof}

This is really straightforward. Let us verify (a) for example. One has

$$\partial_{\xi}^2\left(\int_0^{\xi_1} 1_{\R_+}(\xi + \alpha) d\alpha \right) =
\partial_{\xi} \left(\int_0^{\xi_1}\delta_0(\xi+\alpha) d \alpha\right) =
\partial_{\xi} \left(\int_{\xi}^{\xi+ \xi_1}\delta_0(\alpha) d \alpha\right) = \delta_0(\xi +\xi_1) - \delta_0(\xi).
$$

\end{proof}

To prove now the estimates in Lemma \ref{coef}, we rewrite (for instance) $\widetilde{C^k_{n,n_1}}$ as

$$\frac{1}{2^k}\frac{1}{2^k}
\int_{\R^2}
\left[\int_0^1 1_{\R_+}(\xi +\alpha \xi_1) d \alpha\right]
\widehat{\widetilde{\Psi_k}}(\xi)
\widehat{\widetilde{\Psi_k}}(\xi_1)
e^{- 2\pi i \frac{n}{2^k}\xi} e^{- 2\pi i \frac{n_1}{2^k}\xi_1}
d\xi 
d\xi_1 =
$$

$$\int_{\R^2}
\left[\int_0^1 1_{\R_+}(\xi +\alpha \xi_1) d \alpha\right]
\widehat{\widetilde{\Psi}}(\xi)
\widehat{\widetilde{\Psi}}(\xi_1)
e^{- 2\pi i n\xi} e^{- 2\pi i n_1 \xi_1}
d\xi 
d\xi_1 =
$$

$$
\int_{\R^2}
\left[\frac{1}{\xi_1}\int_0^{\xi_1} 1_{\R_+}(\xi +\alpha) d \alpha\right]
\widehat{\widetilde{\Psi}}(\xi)
\widehat{\widetilde{\Psi}}(\xi_1)
e^{- 2\pi i n\xi} e^{- 2\pi i n_1 \xi_1}
d\xi 
d\xi_1 :=
$$

\begin{equation}\label{23}
\int_{\R^2}
\left[\int_0^{\xi_1} 1_{\R_+}(\xi +\alpha) d \alpha\right]
\widehat{\widetilde{\Psi}}(\xi)
\widehat{\widetilde{\widetilde{\Psi}}}(\xi_1)
e^{- 2\pi i n\xi} e^{- 2\pi i n_1 \xi_1}
d\xi 
d\xi_1,
\end{equation}
where $\widehat{\widetilde{\Psi}}(\xi)$, $\widehat{\widetilde{\Psi}}(\xi_1)$ and $\widehat{\widetilde{\widetilde{\Psi}}}(\xi_1)$ are supported away from the origin and are adapted to scale $1$.

The idea is of course to integrate by parts as much as we can in (\ref{23}) and keep track of the upper bounds that one gets in this way. We begin integrating by parts in $\xi$ as much as we can.
Since both $\int_0^{\xi_1} 1_{\R_+}(\xi +\alpha) d \alpha$ and $\widehat{\widetilde{\Psi}}(\xi)$ depend on $\xi$, the $\xi$ derivatives can hit either of the terms. If the derivative hits
twice the term $\int_0^{\xi_1} 1_{\R_+}(\xi +\alpha) d \alpha$ then, because of the previous Lemma \ref{derivative}, the $\xi$ variable dissapears and becomes $-\xi_1$ (notice that $\xi$ cannot be zero in this case) at which point
(\ref{23}) gets simplified into an expression of type

$$\int_{\R}
\widehat{\widetilde{\Psi}}(-\xi_1)
\widehat{\widetilde{\widetilde{\Psi}}}(\xi_1)
e^{- 2 \pi i \xi_1 (n-n_1)}
d\xi_1.$$
But this term can be integrated by parts as many times as we wish and this explains the appearance of the first upper bound for $|\widetilde{C^k_{n,n_1}}|$. If on the contrary, the
$\xi$ derivative didn't hit the term $\int_0^{\xi_1} 1_{\R_+}(\xi +\alpha) d \alpha$ two times yet even after many integration by parts, this means that we already gained a factor of type $\frac{1}{<n>^{\#}}$
at which moment we stop integrating in $\xi$ and start integrating by parts in $\xi_1$. As before, if the $\xi_1$ derivatives hit the term $\int_0^{\xi_1} 1_{\R_+}(\xi +\alpha) d \alpha$
until one reaches $\delta_0(\xi+\xi_1)$ then $\xi_1$ becomes $-\xi$ and after that one integrates by parts a smooth function obtaining an upper bound of type 
$\frac{1}{<n>^{\#}}\frac{1}{<n - n_1>^{\#}} $ which is smaller than the previously discussed one.

If finally, the $\xi_1$ derivative didn't hit $\int_0^{\xi_1} 1_{\R_+}(\xi +\alpha) d \alpha$ until it becomes $\delta_0(\xi+\xi_1)$, then this means that it keeps hitting the smooth function of $\xi_1$ in which case we obtain
an upper bound of type $\frac{1}{<n>^{\#}}\frac{1}{<n_1>^{\#}} $ as desired.

The second term $C^k_{n, n_1}$ can be treated similarly. One should just remark that in this case the equality $\xi_1= - \xi$ is impossible and only $\delta_0(\xi)$ remains after integrating by parts, 
which explains the slight difference between the two upper bounds.

We are as a consequence left with proving Theorem \ref{discrete}.

\section{Proof of Theorem \ref{discrete}}

The proof is based on the method introduced in \cite{mptt:biparameter} and \cite{mptt:multiparameter}.

Let us assume without loss of generality that the families $(\Phi^2_{I_{\n_2}})_I$ and $(\Phi_I^3)_I$ are of  $\psi$ type (since all the other possible cases can be treated in a similar way).
Fix also $1 < p,q <\infty$ and $0<r<\infty$ so that $1/p+1/q=1/r$. We will prove that $T_{\I}$ maps $L^p\times L^q \rightarrow L^{r,\infty}$ since then (\ref{11}) follows easily by standard interpolation arguments.

As usual, it is enough to show that given a measurable set $E\subseteq \R$ with $|E|=1$, one can find $E'\subseteq E$ with $|E'|\sim 1$ and so that

\begin{equation}\label{24}
\sum_{I\in\I}
\frac{1}{|I|^{1/2}}
|\langle f, \Phi^1_{I_{\n_1}} \rangle|
|\langle g, \Phi^2_{I_{\n_2}} \rangle|
|\langle h, \Phi^3_I \rangle| \lesssim (\log <\n_1>) (\log <\n_2>) 
\end{equation}
where $h:= \chi_{E'}$.

Define now the shifted maximal operator $M^{\n_1}$ and the shifted square function $S^{\n_2}$ as follows.

$$M^{\n_1} f(x) := \sup_{x\in I}\frac{1}{|I|} \int_{\R} |f(y)| \widetilde{\chi_{I_{\n_1}}}(y) d y$$
where $\widetilde{\chi_{I_{\n_1}}}(y)$ denotes the function

$$ \widetilde{\chi_{I_{\n_1}}}(y) = \left(1 + \frac{\dist(y, I_{\n_1})}{|I_{\n_1}|}\right)^{- 100},$$
while $S^{\n_2}$ is given by

$$S^{\n_2} g(x) := \left(\sum_I \frac{|\langle g, \Phi^2_{I_{\n_2}}\rangle|^2}{|I|} 1_{I}(x) \right)^{1/2}.$$

As we will see later on, both of them are bounded on every $L^p$ space for $1 < p < \infty$, with a bound of type $O(\log <\n_1>)$ and $O(\log <\n_2>)$ respectively. Using these two facts we define an exceptional set as follows.

First, define the set $\widetilde{\Omega_0}$ by

$$\widetilde{\Omega_0} := \left\{ x : M^{\n_1} f(x) > C \log <\n_1>  \right\} \cup \left\{ x : S^{\n_2} f(x) > C \log <\n_2>  \right\}.$$

Let now $d$ a positive integer and $\#$ an integer so that $2^d < |\#|\leq 2^{d+1}$. Define the set $\Omega^d_{\#}$ by

$$\Omega^d_{\#} := \left\{ x : M^{\n_1-\#} f (x) > C \log <\n_1-\#> 2^{5 d} \right\}$$
and then the set $\widetilde{\widetilde{\Omega_0}}$ by

$$\widetilde{\widetilde{\Omega_0}} := \bigcup_{d\geq 0} \bigcup_{2^d < |\#|\leq 2^{d+1}}\Omega^d_{\#}.$$
Define also the set $\widetilde{\widetilde{\widetilde{\Omega_0}}}$ in a similar way to $\widetilde{\widetilde{\Omega_0}}$ but by using the function $g$ and the corresponding index $\n_2$ instead. Then, define $\Omega_0$ to be

\begin{equation}
\Omega_0 := \widetilde{\Omega_0} \cup \widetilde{\widetilde{\Omega_0}} \cup \widetilde{\widetilde{\widetilde{\Omega_0}}}
\end{equation}
and finally, the exceptional set

$$\Omega := \left\{ x : M(1_{\Omega_0})(x) > \frac{1}{100} \right\}.$$

Observe that $|\Omega| << 1$ is $C$ is chosen large enough and this allows us to define the set $E'$ by $E' := E\setminus \Omega$ and to observe that $|E'|\sim 1$ as desired.
To be able to estimate (\ref{24}) properly, we split is into two parts as follows:

\begin{equation}
\sum_{I\cap \Omega^c\neq \emptyset}  + \sum_{I\cap \Omega^c = \emptyset} := I + II.
\end{equation}

\subsection*{Estimates for I}

First, we observe that since $I\cap \Omega^c\neq \emptyset$ one has $\frac{|I\cap\Omega_0|}{|I|} < \frac{1}{100}$ which means that $|I\cap\Omega_0^c| > \frac{99}{100} |I|$.

We now perform three independent stopping time type arguments for the functions $f, g, h$ which will be combined carefully later on.

Define first

$$\Omega_1= \left\{ x : M^{\n_1}(f)(x)> \frac{C \log <\n_1> }{2^1} \right\}$$
and set

$$\I_1= \left\{ I \in \I: |I\cap\Omega_1|>\frac{1}{100} |I| \right\},$$
then define

$$\Omega_2= \left\{ x : M^{\n_1}(f)(x)> \frac{C \log <\n_1> }{2^2} \right\}$$
and set

$$\I_2= \left\{ I\in \I \setminus\I_1 : |I\cap\Omega_2|>\frac{1}{100} |I| \right\},$$
and so on. The constant $C>0$ here is the one in the definition of the set $E'$ before. Clearly, since $\I$ is finite, we will run out of dyadic intervals after a while, thus producing the sets $(\{\Omega_n\})_n$
and $(\{\I_n\})_n$.

Independently, define

$$\Omega'_1= \left\{ x : S^{\n_2}(g)(x)> \frac{C \log <n_2> }{2^1} \right\}$$
and set

$$\I'_1= \left\{ I\in \I : |I\cap\Omega'_1|>\frac{1}{100} |I| \right\},$$
then as before define

$$\Omega'_2= \left\{ x : S^{\n_2}(g)(x)> \frac{C \log <n_2> }{2^2} \right\}$$
and set

$$\T'_2= \left\{ I\in \I\setminus\I'_1 : |I\cap\Omega'_2|>\frac{1}{100} |I| \right\},$$
and so on, producing the finitely many sets $(\{\Omega'_n\})_n$ and $(\{\I'_n\})_n$.
Of course, we would like to have such a decomposition available for $h$ as well. To do this, we first need to
construct the analogue of the set $\Omega_0$, for it. To do this, first pick  $N>0$ a big enough integer such that for every
$I\in\I$ we have $|I\cap\Omega^{''c}_{-N}|> \frac{99}{100} |I|$ where we defined

$$\Omega''_{-N}= \left\{ x : S(h)(x)> C 2^N \right\}.$$
Then, similarly to the previous algorithms, we define

$$\Omega''_{-N+1}= \left\{ x : S(h)(x)> \frac{C 2^N}{2^1} \right\}$$
and set

$$\I''_{-N+1}= \left\{ I\in \I : |I\cap\Omega''_{-N+1}|>\frac{1}{100} |I| \right\},$$
then define

$$\Omega''_{-N+2}= \left\{ x : S(h)(x)> \frac{C 2^N}{2^2} \right\}$$
and set

$$\T''_{-N+2}= \left\{ I\in \I \setminus\I''_{-N+1} : |I\cap\Omega''_{-N+2}|>\frac{1}{100} |I| \right\},$$
and so on, constructing the finitely many sets $(\{\Omega''_n\})_n$ and $(\{\T''_n\})_n$.

Using all these decompositions, we can decompose term $I$ further as

\begin{equation}\label{25}
\sum_{l_1, l_2>0, l_3>-N} \sum_{I\in \I_{l_1, l_2, l_3}}
\frac{1}{|I|^{3/2}}
|\langle f, \Phi^1_{I_{\n_1}}\rangle|
|\langle g, \Phi^2_{I_{\n_2}}\rangle|
|\langle h, \Phi^3_{I}\rangle| |I|,
\end{equation}
where $$\I_{l_1, l_2, l_3}:= \I_{l_1}\cap \I'_{l_2}\cap \I''_{l_3}.$$ 
Then, observe that since $I$ belongs to
$\I_{l_1, l_2, l_3}$ this means in particular that $I$ has not been selected at either of the previous $l_1 -1$, $l_2 -1$ and
$l_3 -1$ steps respectively, which means that all of $|I\cap\Omega_{l_1-1}|$,
$|I\cap\Omega'_{l_2-1}|$ and $|I\cap\Omega''_{l_3-1}|$ are smaller than $\frac{1}{100} |I|$ or equivalently, that one has
 
$$|I\cap\Omega^c_{l_1-1}|>\frac{99}{100} |I|,$$
$$|I\cap\Omega^{'c}_{l_2-1}|>\frac{99}{100} |I|$$ 
and
$$|I\cap\Omega^{''c}_{l_3-1}|>\frac{99}{100} |I|,$$
which implies that

\begin{equation}
|I\cap\Omega^c_{l_1-1}\cap\Omega^{'c}_{l_2-1}\cap\Omega^{''c}_{l_3-1}|> \frac{97}{100}|I|.
\end{equation}
Using this in (\ref{25}) one can estimate that expression by

$$\sum_{l_1, l_2>0, l_3>-N} \sum_{I\in \I_{l_1, l_2, l_3}}
\frac{1}{|I|^{3/2}}
|\langle f, \Phi^1_{I_{\n_1}}\rangle|
|\langle g, \Phi^2_{I_{\n_2}}\rangle|
|\langle h, \Phi^3_{I}\rangle| |I\cap\Omega^c_{l_1-1}\cap\Omega^{'c}_{l_2-1}\cap\Omega^{''c}_{l_3-1}|=$$

$$\sum_{l_1, l_2>0, l_3>-N} \int_{\Omega^c_{l_1-1}\cap\Omega^{'c}_{l_2-1}\cap\Omega^{''c}_{l_3-1}}
\sum_{I\in \I_{l_1, l_2, l_3}}
\frac{|\langle f, \Phi^1_{I_{\n_1}}\rangle|}{|I|^{1/2}}
\frac{|\langle g, \Phi^2_{I_{\n_2}}\rangle|}{|I|^{1/2}}
\frac{|\langle h, \Phi^3_{I}\rangle|}{|I|^{1/2}}\chi_{I}(x)\, dx $$

$$\lesssim \sum_{l_1, l_2>0, l_3>-N} \int_{\Omega^c_{l_1-1}\cap\Omega^{'c}_{l_2-1}\cap\Omega^{''c}_{l_3-1}\cap
\Omega_{\I_{l_1, l_2, l_3}}}
M^{\n_1}(f)(x) S^{\n_2}(g)(x) S(h)(x)\, dx$$

\begin{equation}\label{26}
\lesssim \sum_{l_1, l_2>0, l_3>-N} \log <\n_1> \log <\n_2> 2^{-l_1} 2^{-l_2} 2^{-l_3} |\Omega_{\I_{l_1, l_2, l_3}}|,
\end{equation}
where

$$\Omega_{\I_{l_1, l_2, l_3}}:= \bigcup_{I\in\I_{l_1, l_2, l_3}} I.$$
On the other hand, we also have

$$|\Omega_{\I_{l_1, l_2, l_3}}|\leq |\Omega_{\I_{l_1}}|\leq
\left|\left\{ x : M(\chi_{\Omega_{l_1}})(x)> \frac{1}{100} \right\}\right|$$

$$\lesssim |\Omega_{l_1}|= \left|\left\{ x : M^{\n_1}(f)(x)>\frac{C \log <n_1> }{2^{l_1}} \right\}\right|\lesssim 2^{l_1 p}.$$
Similarly, we have

$$|\Omega_{\I_{l_1, l_2, l_3}}|\lesssim 2^{l_2 q}$$
and also

$$|\Omega_{\I_{l_1, l_2, l_3}}|\lesssim 2^{l_3 \alpha},$$
for every $\alpha >1$. Here we used the fact that all the operators $M^{\n_1}$, $S^{\n_2}$ and $S$ are bounded
on $L^s$ as long as $1<s< \infty$ and also that $|E'_3|\sim 1$.
In particular, this implies that

\begin{equation}\label{27}
|\Omega_{\I_{l_1, l_2, l_3}}|\lesssim 2^{l_1 p \theta_1}
 2^{l_2 q \theta_2} 2^{l_3 \alpha \theta_3}
\end{equation}
for any $0\leq \theta_1, \theta_2, \theta_3 < 1$, such that $\theta_1+ \theta_2 +\theta_3= 1$.

On the other hand, (\ref{26}) can be split into

\begin{equation}\label{28}
\log<\n_1> \log<\n_2> \left(\sum_{l_1, l_2>0, l_3>0} 2^{-l_1} 2^{-l_2} 2^{-l_3} |\Omega_{\I_{l_1, l_2, l_3}}|+ 
\sum_{l_1, l_2>0, 0>l_3>-N} 2^{-l_1} 2^{-l_2} 2^{-l_3} |\Omega_{\I_{l_1, l_2, l_3}}|\right).
\end{equation}
To estimate the first expression in (\ref{28}) we use the inequality (\ref{27}) for $\theta_1, \theta_2, \theta_3$ so that
$1-p \theta_1 > 0$, $1-q\theta_2 >0$ and $1-\alpha\theta_3 > 0$, while to estimate the second term we use (\ref{27}) for $\theta_1, \theta_2, \theta_3$
such that $1-p\theta_1>0$, $1-q\theta_2>0$ and $1 - \alpha\theta_3 < 0$. With these choices, the sum in (\ref{28}) is indeed 
is $O(\log <\n_1> \log <\n_2>  )$ as desired. This ends the discussion of $I$.

\subsection*{Estimates for II}

This term is simpler to estimate, now that we defined our exceptional set so carefully. Notice that the intervals of interest are those inside $\Omega$.
One can split them as $\bigcup_{d\geq 0} \I_d$ where

$$\I_d := \left\{ I\in \I : I\subseteq \Omega \,\,\text{and}\,\, 2^d \leq \frac{\dist(I,\Omega^c)}{|I|}  < 2^{d+1} \right\}.$$
Observe that for any $d\geq 0$ one has

$$\sum_{I\in\I_d} |I| \lesssim |\Omega| \lesssim 1.$$
Also, for every $I\in\I_d$ one has that $2^d I \cap \Omega^c = \emptyset$ and also there exists $\widetilde{I}$ dyadic and of the same length, which lies $\#$ steps of length $|I|$ away from $I$ (with $2^d\leq |\#|\leq 2^{d+1}$),
and having the property that $\widetilde{I}\cap \Omega^c \neq \emptyset$. In particular, this means that $I_{\n_1}$ and $I_{\n_2}$ are $\n_1-\#$ and $\n_2 - \#$ steps of length $|I|$ away from $\widetilde{I}$.
Using all these facts, one can estimate term $II$ by

\begin{equation}
\sum_{d\geq 0} \sum_{I\in\I_d}
\frac{|\langle f, \Phi^1_{I_{\n_1}}\rangle|}{|I|^{1/2}}
\frac{|\langle g, \Phi^2_{I_{\n_2}}\rangle|}{|I|^{1/2}}
\frac{|\langle h, \Phi^3_{I}\rangle|}{|I|^{1/2}} |I| \lesssim
\end{equation}

$$ \sum_{d\geq 0} \sum_{ 2^d\leq |\#|\leq 2^{d+1} }\sum_{I\in\I_d} (\log <\n_1 - \#>) 2^{ 5d} (\log <\n_2 - \#>) 2^{ 5d} 2^{-M d} |I| \lesssim $$

$$(\log <\n_1>) (\log <\n_2 >)$$
by using the trivial fact that 

$$\log <\n_j - \#> \leq \log <\n_j> <\#>$$ 
for $j=1,2$.

The proof is now complete.

\section{Appendix $1$: Logarithmic estimates for the shifted maximal function}

The goal of this section is to prove the following theorem that have been used before. This result can be found in Stein \cite{stein}, but we decided to give a selfcontained proof of it here (which we (re)discovered independently), 
not only for reader's convenience, but also for the fact that some particular notations that will be introduced, will be quoted and used later on as well. 
\begin{theorem}\label{maximal} (\cite{stein})
For any $\n\in\Z$, the shifted maximal function $M^{\n}$ is bounded on every $L^p$ space for $1<p<\infty$, with a bound of type $O(\log <\n>)$.
\end{theorem}

\begin{proof}
First, we observe that in order to prove the desired estimates, it is enough to prove them for the corresponding {\it sharp} maximal function $\widetilde{M^{\n}}$ defined by

\begin{equation}
\widetilde{M^{\n}} f(x) := \sup_{x\in I} \frac{1}{|I_{\n}|} \int_{I_{\n}} |f(y)| d y
\end{equation}
where the suppremum is taken only over dyadic intervals. 

To see this, fix $x$ and $I$ so that $x\in I$. One can write

$$\frac{1}{|I_{\n}|} \int_{I_{\n}} |f(y)| d y \lesssim \sum_{\#\in \Z}\left[\frac{1}{|I^{\#}_{\n}|} \int_{I^{\#}_{\n}} |f(y)| d y\right] \frac{1}{<\#>^{100}},$$
where $I^{\#}_{\n}$ is the dyadic interval of the same length with $I_{\n}$ and lying $\#$ steps of length $|I_{\n}|$ away from it.
In particular, using the above and assuming that the theorem holds for $\widetilde{M^{\n}}$, one has

$$\|M^{\n} f\|_p \lesssim \sum_{\#\in \Z}\frac{1}{<\#>^{100}} \|\widetilde{M^{\n + \#}}f\|_p \lesssim \sum_{\#\in \Z}\frac{1}{<\#>^{100}} (\log <\n + \#>)  \|f\|_p \lesssim$$

$$\lesssim \sum_{\#\in \Z}\frac{1}{<\#>^{100}} (\log (<\n><\#>)) \lesssim \log <\n>  \|f\|_p,$$
as desired. We are then left with proving the theorem for $\widetilde{M^{\n}}$.

Let now $\lambda > 0$. We claim that one has the following inequality

\begin{equation}\label{star}
| \{ x : \widetilde{M^{\n}} f (x) > \lambda \} | \lesssim (\log <\n>) |\{ x : M f (x) > \lambda \} |
\end{equation}
where $M$ is the classical Hardy-Littlewood maximal operator. Assuming (\ref{star}), the theorem  for $\widetilde{M^{\n}}$ follows from the Hardy-Littlewood theorem by interpolation with the trivial
$L^{\infty}$ estimate.

To finally prove (\ref{star}) denote by $\I_{\n}^{\lambda}$ the collection of all dyadic and maximal with respect to inclusion intervals $I_{\n}$, for which $$\frac{1}{|I_{\n}|} \int_{I_{\n}} |f(y)| d y > \lambda.$$
Note that all of them are disjoint and one also has

$$\bigcup_{I_{\n}\in \I_{\n}^{\lambda}} I_{\n} = \{ x : Mf(x) >\lambda \}.$$
Then, for every such a selected maximal dyadic interval $I_{\n}$, consider its dyadic subintervals of length $|I_{\n}|$, $|I_{\n}|/2$, $|I_{\n}|/2^2$ ... , etc. Observe that there exsist only $[\log <\n>]$ disjoint dyadic intervals 
$I_{\n}^1, I_{\n}^2, ... ,I_{\n}^{[\log<\n>]}$ of the same length with $|I_{\n}|$, so that 
the translate with $ - \n$ corresponding units of any such smaller dyadic subinterval of $I_{\n}$, becomes a subinterval of one of these $I_{\n}^1, I_{\n}^2, ... ,I_{\n}^{[\log<\n>]}$. The claim is now that

$$\{ x : \widetilde{M^{\n}} f(x) > \lambda \} \subseteq \bigcup_{I_{\n}\in \I_{\n}^{\lambda}} (I_n \cup I_{\n}^1\cup ... \cup I_{\n}^{[\log<\n>]}).$$
To see this, pick $x_0$ so that $M^{\n} f(x_0) > \lambda$. Then, this means that there exists a dyadic interval $J$ containing $x_0$, so that $\frac{1}{|J_{\n}|} \int_{J_{\n}} |f(y)| d y > \lambda$.
Because of the previous construction, one can for sure find one selected maximal interval of type $I_{\n}$, so that $J_{\n}\subseteq I_{\n}$. But then, this means in particular that $J$ itself will be a subset of either
$I_{\n}$ or $I_{\n}^1$ or ... or $I_{\n}^{[\log<\n>]}$ which implies the claim.

It is now easy to see that this claim together with the disjointness of the maximal intervals $I_{\n}$, imply (\ref{star}). The proof is then complete. \footnote{Of course, since the trivial $L^{\infty}$ estimate comes with an $O(1)$ bound, by interpolation the $L^p$ operatorial bound of $M^{\n}$ will be even $O((\log <\n>)^{1/p})$. 
But for simplicity, we used the $O(\log <\n>)$ bound all the time.       }

\end{proof}

\section{Appendix $2$: Logarithmic estimates for the shifted square function}

The goal of this last section is to prove the following theorem which played an important role earlier in the argument \footnote{It may very well be that this result has been observed before (as it was the case with the previous shifted
maximal function) but since we didn't find it in the literature, we include a selfcontained proof of it in what follows.}.

\begin{theorem}\label{square}
For any $\n\in\Z$, the shifted square function $S^{\n}$ is bounded on every $L^p$ space for $1<p<\infty$, with a bound of type $O(\log <\n>)$.
\end{theorem}

\begin{proof}

Besides the observations of the previous section, the proof is based on a classical Calder\'{o}n-Zygmund decomposition \cite{stein}.

First, let us observe that $S^{\n}$ is bounded on $L^2$ with a bound independent of $\n$. Indeed, one can see that

$$\|S^{\n} f \|_2 = \left(\sum_I \langle f, \Phi_{I_{\n}} \rangle ^2 \right)^{1/2}$$
which is clearly comparable to the $L^2$ norm of the classical Littlewood-Paley square function, which is known to be bounded on $L^2$.

Next, we show that 

\begin{equation}\label{weak}
\|S^{\n} f \|_{1,\infty} \lesssim (\log <\n>) \|f\|_{1}.
\end{equation}
or more specifically that

\begin{equation}\label{weak1}
|\{ x\in\R : S^{\n} f (x) > \lambda \} | \lesssim \log <\n> \frac{1}{\lambda} \|f\|_1.
\end{equation}
Fix such a $\lambda > 0$ and perform a Calder\'{o}n-Zygmund decomposition of the function $f$ {\it at level $\lambda$}. Pick one by one maximal dyadic intervals $J$ so that

$$\frac{1}{|J|}\int_J |f(y)| dy > \lambda.$$
Observe that all these intervals are by construction disjoint and denote their union with $\Omega$. One has

\begin{equation}\label{omega}
|\Omega| = \sum_J |J| < \frac{1}{\lambda} \sum_J \int_J |f(y)| dy \leq \frac{1}{\lambda} \|f\|_1.
\end{equation}
Split now the function $f$ as 

$$f = g + b$$
where

$$g := f\chi_{\Omega^c} + \sum_J \left[ \frac{1}{|J|}\int_J f(y) dy  \right] \chi_J$$
and

$$b := f - g := \sum_J b_J$$
where

$$b_J := \left[f - \frac{1}{|J|}\int_J f(y) dy \right] \chi_J.$$
Clearly, $\supp b_J \subseteq J$. Observe also that one has

$$|f(x)|\leq \lambda$$
for every $x\in \Omega^c$ and as a consequence,

$$\|g\|_{\infty} \lesssim \lambda$$
since one also observes that

$$|\frac{1}{|J|}\int_J f(y) dy| \leq \frac{1}{|J|}\int_J |f(y)| dy \leq \frac{2}{|\widetilde{J}|} \int_{\widetilde{J}} |f(y)| d y \leq 2 \lambda$$
where $\widetilde{J}$ is the unique dyadic interval containing $J$ and twice as long. It is also important to observe that

$$\int_{\R} b_J(y) d y = 0$$
by definition and also that

$$\|b_J\|_1 = \int_J |b_J(y)| d y \leq \int_J |f(y)| d y + \left(\frac{1}{|J|}\int_J |f(y)| dy\right) |J| \lesssim \int_J |f(y)| d y  \lesssim \lambda |J|$$
as we have seen.

Using all these properties, one can write

\begin{equation}\label{weak2}
|\{ x\in\R : S^{\n} f (x) > \lambda \} | \leq |\{ x\in\R : S^{\n} g (x) > \lambda/2 \} | + |\{ x\in\R : S^{\n} b (x) > \lambda/2 \} |.
\end{equation}
To estimate the first term in (\ref{weak2}), we use the $L^2$ boundedness of $S^{\n}$ and we write

$$|\{ x\in\R : S^{\n} g (x) > \lambda/2 \} | \lesssim \frac{1}{\lambda^2} \|S^{\n} g \|_2^2 \lesssim $$

$$\frac{1}{\lambda^2} \|g\|_2^2 = \frac{1}{\lambda^2} \int_{\R}|g(x)|^2 d x \lesssim \frac{1}{\lambda^2} \lambda \int_{\R}|g(x)| d x = $$

$$\frac{1}{\lambda} \|g\|_1 \lesssim \frac{1}{\lambda} \left( \int_{\Omega^c}|f(x)| d x + \sum_J \int_J |f(x)| d x \right) \lesssim \frac{1}{\lambda} \|f\|_1,$$
as desired.

To estimate the second term in (\ref{weak2}), we proceed as follows. First, for any interval $J$, consider the associated $J^1, J^2, ..., J^{[\log<\n>]}$ as defined in the previous section and define the set $\Omega_J$ by

$$\Omega_J := 5 J\cup 5 J^1 \cup 5 J^2 \cup ... \cup 5 J^{[\log<\n>]}.$$
Then, one has

\begin{equation}\label{weak3}
|\{ x\in\R : S^{\n} b (x) > \lambda/2 \} | \leq \left|\left\{ x\in \bigcup_J \Omega_J : S^{\n} b (x) > \lambda/2 \right\}\right| + 
\end{equation}

$$\left|\left\{ x\in (\bigcup_J \Omega_J)^c : S^{\n} b (x) > \lambda/2 \right\}\right|.$$

The first expression is easy to estimate since one can write

$$\left|\left\{ x\in \bigcup_J \Omega_J : S^{\n} b (x) > \lambda/2 \right\}\right| \leq \left|\bigcup_J \Omega_J\right| \lesssim (\log <\n>) \sum_J |J| \lesssim (\log<\n>) \frac{1}{\lambda} \|f\|_1,$$
as we have seen before. The second expression in (\ref{weak3}) can be majorized by

$$\frac{1}{\lambda} \int_{(\bigcup_J \Omega_J)^c} S^{\n} b (x) d x \leq \frac{1}{\lambda} \sum_J \int_{(\bigcup_J \Omega_J)^c} S^{\n} b_J (x) d x \leq \frac{1}{\lambda} \sum_J \int_{(\Omega_J)^c} S^{\n} b_J (x) d x     $$ 
and we claim now that for any $J$ one has

\begin{equation}\label{l1}
\int_{(\Omega_J)^c} S^{\n} b_J (x) d x \lesssim \lambda |J|.
\end{equation}
Assuming (\ref{l1}), one can continue the previous inequality and further majorize it by

$$\frac{1}{\lambda} \lambda \sum_J |J| \lesssim |\Omega| \lesssim \frac{1}{\lambda} \|f\|_1$$
as desired.

We are then left with proving our claim (\ref{l1}). First, we majorize the left hand side of it by

$$\int_{(\Omega_J)^c} \left(\sum_I \frac{|\langle b_J, \Phi_{I_{\n}}\rangle |}{|I|^{1/2}} 1_I(x) \right) dx = \sum_I \int_{(\Omega_J)^c} \frac{|\langle b_J, \Phi_{I_{\n}}\rangle |}{|I|^{1/2}} 1_I(x) ) dx =$$

$$\sum_{|I|\leq |J|} \int_{(\Omega_J)^c} \frac{|\langle b_J, \Phi_{I_{\n}}\rangle |}{|I|^{1/2}} 1_I(x) ) dx + 
\sum_{|I| > |J|} \int_{(\Omega_J)^c} \frac{|\langle b_J, \Phi_{I_{\n}}\rangle |}{|I|^{1/2}} 1_I(x) ) dx :=$$

$$ A + B.$$

\subsection*{Estimating A}

The main observation here is to realize that since $|I|\leq |J|$ and $I\cap (\Omega_J)^c \neq \emptyset$, one must in particular have $I_{\n}\cap 3J = \emptyset$. This allows one to estimate $A$ by

$$\sum_{|I_{\n}|\leq |J|} \left(1 + \frac{\dist (I_{\n}, J)}{|I_{\n}|}\right)^{-10} \int_{\R} |b_J(y)| d y \lesssim
\lambda |J| \sum_{|I_{\n}|\leq |J|} \left(1 + \frac{\dist (I_{\n}, J)}{|I_{\n}|}\right)^{-10} \lesssim \lambda |J|,$$
as required by (\ref{l1})

\subsection*{Estimating B}

This time, one has to take into account that fact that

\begin{equation}\label{zero}
\int_{\R} b_J(y) d y = 0.
\end{equation}
As before, one can estimate $B$ by

$$\sum_{|I_{\n}| > |J|} |\langle b_J, \Phi^{\infty}_{I_{\n}}\rangle|$$
where this time $\Phi^{\infty}_{I_{\n}} := |I_{\n}|^{1/2}\Phi_{I_{\n}}$ is an $L^{\infty}$ normalized bump. In order to emphasize that the dependence of $\n$ is irrelevant now, we rewrite the above expression as 

$$\sum_{|K| > |J|} |\langle b_J, \Phi^{\infty}_{K}\rangle|$$
where the sum is over dyadic intervals $K$.

Fix $K$ such that $|K| > |J|$ and observe that

$$|\langle b_J, \Phi^{\infty}_{K}\rangle| = \left|\int_{\R} b_J(z)\overline{\Phi^{\infty}_{K}}(z) d z\right| = \left|\int_{J} b_J(z)(\overline{\Phi^{\infty}_{K}}(z) - \overline{\Phi^{\infty}_{K}}(c_J)) d z \right|$$
where $c_J$ denotes the midpoint of the interval $J$.

Then, observe that for $z\in J$, one has

$$|\overline{\Phi^{\infty}_{K}}(z) - \overline{\Phi^{\infty}_{K}}(c_J)| \lesssim |J|\frac{1}{|K|} \left(1 + \frac{\dist(K, J)}{|K|} \right)^{-10}$$
and so the previous term becomes smaller than

$$|J|\frac{1}{|K|} \left(1 + \frac{\dist(K, J)}{|K|} \right)^{-10} \int_J |b_J(y)| d y \lesssim |J|\frac{1}{|K|} \left(1 + \frac{\dist(K, J)}{|K|} \right)^{-10} \lambda |J|.$$
Finally, the corresponding (\ref{l1}) follows from the straightforward observation that

$$\sum_{|K| > |J|} \frac{|J|}{|K|} \left(1 + \frac{\dist(K, J)}{|K|} \right)^{-10} \lesssim 1.$$

By interpolating between $L^2$ and weak-$L^1$ we obtain the theorem for any $1<p\leq 2$. To prove the rest of the estimates we proceed as usual, by duality.
Fix then $2 < p <\infty$. By using Khinchin inequality, one can write

$$\| S^{\n} f \|_p^p = \int_{\R} \left(\sum_I \frac{|\langle f , \Phi_{I_{\n}}\rangle |^2}{|I|} \chi_I (x) \right)^{p/2} dx \lesssim $$

\begin{equation}\label{hincin}
\int_{\R} \int_0^1 \left| \sum_I r_I (t) \langle f , \Phi_{I_{\n}}\rangle h_I (x) \right|^p d x d t = \int_0^1 \left\|\sum_I r_I (t) \langle f , \Phi_{I_{\n}}\rangle h_I \right\|_p^p d t,
\end{equation}
where $(r_I)_I$ are the Rademacher functions and $(h_I)_I$ the $L^2$-normalized Haar functions.

Fix now $t\in [0,1]$ and consider the linear operator 

$$ f \rightarrow \sum_I r_I (t) \langle f , \Phi_{I_{\n}}\rangle h_I.$$
Using the fact that $S^{\n}$ and the Littlewood-Paley square function associated to $(h_I)_I$ are bounded below $L^2$, an argument identical to the one used to prove Theorem \ref{discrete} shows that the above operator is also bounded
below $L^2$ and by duality, above $L^2$ as well, with bounds independent of $t$ which grow logarithmically in $<\n>$. Using this fact in (\ref{hincin}), completes the proof of the theorem.

\end{proof}

\section{Generalizations}

Let us first observe that the first commutator $C_1 f$ can also be written as

\begin{equation}
C_1 f(x) = p.v. \int_{\R} 
\left(\frac{\Delta_t}{t} A(x)\right) f(x+t) \frac{d t}{t}
\end{equation}
where $\Delta_t$ is the {\it finite difference} operator at scale $t$ given by

$$\Delta_t g(x) := g(x+t) - g(x).$$ There is a very simple way to motivate the introduction of this operator. Start with the Leibnitz rule identity

$$(A f)' = A' f + A f'$$
and solve for $A' f$ to obtain

$$A' f = (A f)' - A f' = D(A f) - A D f = [D, A] f$$
where $D$ is the operator of taking one derivative and $A$ is viewed now as the operator of multiplication with the function $A(x)$.
 In particular, assuming that $A'\in L^{\infty}$, the commutator $[D, A]$ maps $L^p$ into itself boundedly, for every $1<p<\infty$.
{\it Does this property hold for the operator $[|D|, A]$ as well ?} one might ask. A straightforward calculation shows that $[|D|, A]$ is precisely the first commutator of Calder\'{o}n.

Given these, it is of course natural to ask {\it what can be said about the double commutator $[|D|, [|D|, A]]$ ?}

A direct calculation shows that the expression $[|D|, [|D|, A]] (f) (x)$ is equal to

\begin{equation}\label{commm1}
 p.v. \int_{\R^2}
\left (\frac{\Delta_t}{t}\circ  \frac{\Delta_s}{s} A(x)\right) f(x+t+s) \frac{ d t}{t} \frac{ d s}{s}
\end{equation}
a formula that can be naturally seen as a bilinear operator, this time depending on $f$ and $A''$. Its symbol can be again calculated easily and it is given by

$$\left(
\int_0^1 {\rm sgn} (\xi + \alpha \xi_1) d \alpha
\right)^2$$
which is precisely the square of the symbol of the first commutator of Calder\'{o}n.

\begin{theorem}\label{gen22}
Let $a\neq 0$ and $b\neq 0$ and consider the expression

$$
 p.v. \int_{\R^2}
\left (\frac{\Delta_{a t}}{t}\circ  \frac{\Delta_{b s}}{s} A(x)\right) f(x+t+s) \frac{ d t}{t} \frac{ d s}{s}.
$$
Viewed as a bilinear operator in $f$ and $A''$, it extends naturally as a bounded operator from $L^p\times L^q$ into $L^r$ for every
$1<p, q \leq \infty$ with $1/p+1/q=1/r$ and $1/2 < r <\infty$.
\end{theorem}
To prove this theorem, one applies the same method described earlier for the first commutator. One just has to observe that the symbol of this operator is given by

$$
\left(\int_0^1 {\rm sgn} (\xi + \alpha  a \xi_1) d \alpha\right)
\left(\int_0^1 {\rm sgn} (\xi + \alpha  b \xi_1) d \alpha\right)
$$
and after that to realize that {\it each factor} satisfies the same desired {\it quadratic estimates}. So this time one needs to decompose each factor as a double Fourier series as we did before. 
The fact that one can go all the way down to $1/2$ with the estimates, is a simple
consequence of the statement that series of type

$$\sum_{n_1, n_2\in \Z} |C(n_1, n_1)|^r \log <n_1> \log <n_2>$$
are always convergent as long as the constants $C(n_1, n_2)$ decay at least quadratically in $n_1$ and $n_2$ and $r > 1/2$. The details are straightforward and are left to the reader. And clearly, one can generalize the above theorem even further, in the most obvious way. We will come back to this in the second paper of the sequel.

Another generalization we have in mind comes from the following identity

\begin{equation}\label{identity}
A' B' = (AB)'' - (B A')' - (A B')' + A' B'.
\end{equation}
As a consequence of it, the right hand side of \eqref{identity} satisfies H\"{o}lder estimates of type

$$\|(AB)'' - (B A')' - (A B')' + A' B' \|_r \lesssim \|A'\|_p \|B'\|_q$$
for indices $p, q, r$ as before. {\it Does this inequality continue to hold if one replaces every derivative $D$ by its modulus $|D|$ ?}
As before, a direct calculation shows that the new expression

$$|D|^2 (AB) - |D|( B |D| A) - |D|(A |D| B) + (|D| A) (|D| B)$$
can be rewritten as

\begin{equation}\label{circular1}
 p.v. \int_{\R^2}\left (\frac{\Delta_t}{t} A(x+s)\right)
\left (\frac{\Delta_s}{s} B(x+t)\right) 
\frac{ d t}{t} \frac{ d s}{s}.
\end{equation}
The right way to look at this formula is to view it as a bilinear operator in $A'$ and $B'$. Its symbol can be calculated quite easily and it is given by

\begin{equation}
\left(\int_0^1 {\rm sgn} (\xi_1 + \alpha \xi_2) d \alpha\right)
\left(\int_0^1 {\rm sgn} (\xi_2 + \beta \xi_1) d \beta\right)
\end{equation}
which is a symmetric function in the variables $\xi_1$ and $\xi_2$. Because of this symmetry we like to call  expressions such as the ones in (\ref{circular1}) {\it circular commutators}. We will return to them in the second paper of the sequel.

\begin{theorem}\label{gen33}
Let $a\neq 0$ and $b\neq 0$ and consider the expression

$$
 p.v. \int_{\R^2}\left (\frac{\Delta_{ a t}}{t} A(x+s)\right)
\left (\frac{\Delta_{b s}}{s} B(x+t)\right) 
\frac{ d t}{t} \frac{ d s}{s}.
$$

Viewed as a bilinear operator in $A'$ and $B'$, it extends naturally as a bounded operator from $L^p\times L^q$ into $L^r$ for every
$1<p, q \leq \infty$ with $1/p+1/q=1/r$ and $1/2 < r <\infty$.
\end{theorem}
The proof uses the same method, since it is not difficult to see that the symbols of such bilinear operators are again products of symbols of the first commutator kind and they each satisfy the same {\it quadratic estimates}.

\end{document}